\documentclass{article}
\usepackage{latexsym,amssymb,amsmath}
 \newtheorem{thm}{Theorem}[section]
 \newtheorem{cor}[thm]{Corollary}
 \newtheorem{lem}[thm]{Lemma}
 \newtheorem{prop}[thm]{Proposition}
 \newtheorem{ex}[thm]{Example}
 \newtheorem{defn}[thm]{Definition}
 
 \newcommand{\reals}{\mathbb{R}}

 \newenvironment{proof}{{\bf Proof.}\rm  }{\hfill $\Box$}{\vskip5mm}
\def\ds{\mathcal{DS}\big(\ell^p(I)\big)}
\def\ds{\mathcal{DS}\big(\ell^p(I)\big)}

\def\ds{\mathcal{DS}\big(\ell^p(I)\big)}
\def\dslpn{\mathcal{DS}\big(\ell^p(\mathbb N)\big)}
\def\lpi{\ell^p(I)}
\def\lpj{\ell^p(J)}
\def\lpn{\ell^p(\mathbb N)}
\def\ifn{I_f^n}
\def\ifp{I_f^+}
\def\ifm{I_f^-}
\def\igp{I_g^+}

\def\ifk{I_f^k}

\def\ign{I_g^n}

\def\a{\alpha}
\def\b{\beta}
\def\t{\theta}
\def\s{\sigma}
\def\ds{\mathcal{DS}\big(\ell^p(I)\big)}
\def\Plpi{\mathcal P\big(\lpi\big)}
\def\S{\v\sum}
\def\dij{d_{ij}}
\def\tij{\langle Te_j,e_i\rangle}
\def\Mpr{\mathcal M_{\!_{Pr}}\!\big(\lpi\big)}

\def\v{\displaystyle}
\def\r{\rightarrow}

\begin{document}

\title{Linear Preservers of Majorization on $\lpi$
}
\author{ F. Bahrami,\ \  A. Bayati Eshkaftaki,\ \  S.M. Manjegani \thanks{This work is partially supported by a grant from Isfahan
University of Technology. }}
\maketitle
\vspace{-.5cm}
\begin{center}
{\small Department of Mathematical Sciences, Isfahan University of
Technology\\ Isfahan 84156, IRAN\footnote{ email addresses. fbahrami@cc.iut.ac.ir\ \ a.bayati@math.iut.ac.ir\ \ manjgani@cc.iut.ac.ir} }
\end{center}

\hrule
\begin{abstract}
In this paper, using doubly stochastic operators, we extend the notion of majorization to the space $\ell^p(I)$, where $I$ is assumed to be an infinite set, and then,  in the case $p\in (1,\infty)$,
characterize the structure of all bounded linear maps on this space which preserve majorization.
\\
\end{abstract}
\hrule
\section{Introduction}

Majorization in finite dimension has been widely studied  as a result of its applications to many areas of mathematics, such as matrix analysis, operator theory, frame theory, and inequalities involving convex functions, as well as other sciences like physics and economics. See, for example, the papers \cite{AK2005}, \cite{CRP2006}, \cite{NV2001} and \cite{P2012}. We also  refer the reader to the standard text by Marshall and Olkin \cite{MO1979}. For a pair of vectors $x$ and $y$ in $\mathbb R^n$, $x$ is called majorized by $y$, denoted by $x\prec y$, if
\[ \sum_{i=1}^k x_i^{\downarrow} \leq \sum_{i=1}^k y_i^{\downarrow} \qquad (k=1,2,\ldots,n)\]
and
\[ \sum_{i=1}^n x_i^{\downarrow}= \sum_{i=1}^n y_i^{\downarrow}\]
where $x_1^{\downarrow}\geq x_2^{\downarrow}\geq\cdots\geq x_n^{\downarrow}$ is the decreasing rearrangement of  components of a vector $x$.

There are some equivalent conditions for vector majorization. For example, Hardy, Littlewood and Polya \cite{HLP1932} proved   that $x\prec y$ if and only if $x=Dy$ for some doubly stochastic matrix $D$.  We recall that a square matrix with non-negative real entries is called doubly stochastic if each of its row sums and column sums equal 1. As we will see in Section 3, this equivalent condition will serve as our motivation to define majorization on certain spaces other than $\mathbb R^n$.

In more recent years the extension of majorization theory to infinite sequences has turned up and obtained some applications (see for example \cite{KW2010}). In this paper, we will consider majorization on the space $\lpi$, for $1\leq p<+\infty$, and in the case where $I$ is an infinite set. Our main interest is in linear maps which preserve majorization. The following result, due to Ando, characterizes these maps in finite dimension.
\begin{thm}\label{thm1.1} {\rm \cite{A1989}}. Let $T:\mathbb R^n\rightarrow\mathbb R^n$  be a linear map. Then $T(x)\prec T(y)$ whenever $x\prec y$ (i.e. T preserves majorization) if and only if one of the following conditions hold.\\
{\rm (i)} $T(x)={\rm tr}(x) a$, for some $a\in \mathbb R^n$.\\
{\rm (ii)} $T(x)=\beta P(x)+\gamma {\rm tr}(x)e$ for some $\beta,\gamma\in \mathbb R$ and permutation $P:\mathbb R^n\rightarrow \mathbb R^n$.
\end{thm}
Here ${\rm tr}(x)=\sum_{i=1}^n x_i$ is the trace of the vector $x\in \mathbb R^n$. Also $e\in \mathbb R^n$  denotes the vector $(1,1,\ldots,1)$.

 Quit different from this result, our main theorem asserts that if $I$ is an infinite set and $1<p<+\infty$, then a linear map $T:\lpi\rightarrow \lpi$ preserves majorization if and only if  the columns of $T$ are permutations of each other and in each row of $T$ there is at most one non-zero element. Note that, in condition (ii) of Theorem \ref{thm1.1}, if $\gamma=0$ then the resulted $T$ has the structure mentioned above.

 The organization of the paper is as follows. In the next section we recall the definition of doubly stochastic operators on the space $\lpi$, for $1\leq p<+\infty$.  We obtain some  properties  and give a way of constructing these operators. In section 3, we give a definition of majorization on $\lpi$ based on doubly stochastic operators. The main theorem of this section asserts that if  $f\prec g$ and $g\prec f$, for $f,g\in \lpi$, then there exists a   permutation $P:\lpi\rightarrow \lpi$ such that $f=Pg$, a result which is well-known if $I$ is a finite set. Finally, in the last section we characterize the linear preservers of majorization on $\lpi$, for an infinite set $I$ and in the case where $1<p<+\infty$. We end this section with an example which shows that this characterization is not true for $p=1$.

\section{Doubly Stochastic Operators}
We first recall some definitions. For a non-empty set $I$ and a real $p\in [1,+\infty)$, let
$\ell^p(I)$ be the Banach space of all functions $f:I\rightarrow
\reals$ with
 \[\|f\|_p:=\big(\sum_{i\in
I}|f(i)|^p\big)^{\frac{1}{p}}<+\infty\]
   An element $f\in
\ell^p(I)$ can be represented  as $\sum_{i\in I} f(i)\, e_i$,
where $e_i:I\rightarrow \reals$ is defined by
$e_i(j)=\delta_{ij}$, the Kroneker delta. Considering $e_i$ as an
element of the dual space  of $\ell^p(I)$, we have
\[ \forall i\in I\qquad f(i)=\langle f,e_i\rangle\]
where $\langle\cdot,\cdot\rangle$ stands for the dual pairing.
Hence, for $f\in \ell^p(I)$ we will have the representation
\[ f=\sum_{i\in I} \langle f,e_i\rangle e_i\]

 It is a well-known fact that $\ell^p(I)$ is  an ordered vector space
 (and, in fact, a Banach lattice)
  under the natural partial ordering
  on the set of real valued
  functions defined on $I$.
We recall that a linear operator $A$ on an ordered vector space  $X$ is
called positive if $Ax\geq 0$ whenever $x\geq 0$.
\begin{defn}\label{defn1.1}
Let $I$ and $J$ be two non-empty sets, and suppose  $A:\lpj\r\lpi$
is a bounded linear operator. Then $A$ is called
 \begin{itemize}
 \item[{\rm(i)}]row stochastic
(respectively, column stochastic) if $A$ is positive and
\begin{equation}\label{defn1.1eq1}
 \forall i\in I,\quad\sum_{j\in J} \langle Ae_j,e_i\rangle =1\quad
 \Big(\forall j\in J,\quad \sum_{i\in I}\langle Ae_j,e_i\rangle =1\Big)
  \end{equation}
 \item[{\rm(ii)}] doubly stochastic
if $A$ is both row and column stochastic.
 \item[{\rm(iii)}] a permutation if there exists a bijection $\t: J\r I$ for which $Ae_j=e_{\t
(j)}$, for each $j\in J$.
 \end{itemize}
\end{defn}

As the following theorem shows if there exists  a doubly stochastic operator between the spaces $\lpi$ and $\lpj$ then $I$ and $J$ have the same cardinality. This result plays a crucial role in the proof of the main theorem of Section 3.
\begin{thm}\label{|I|=|J|}
Let $I,J$ be two arbitrary non-empty sets. Then there exists a
doubly stochastic operator $D:\lpj\r\lpi$ if and only if
$|J|=|I|$, where $|I|$ denotes the cardinal number of a set $I$.
\end{thm}
\begin{proof}
First, suppose there exists a doubly stochastic operator
$D:\lpj\r\lpi$. Using the relation
\begin{equation*}
\S_{j\in J}1=\sum_{j\in J}\sum_{i\in I}\langle De_j,e_i\rangle =
\sum_{i\in I}\sum_{j\in J}\langle De_j,e_i\rangle  = \sum_{i\in I}1,
\end{equation*}
$J$ is finite if and only if $I$ is finite, and in this case $|I|=|J|$.

Now suppose $J$ is infinite. Let
$$C=\{(i,j)\in I\times J\ ;\
\langle De_j,e_i\rangle>0\}.$$
Then $C=\bigcup_{i\in I}(\{i\}\times C_i)=\bigcup_{j\in
I}(C^j\times \{j\})$, where $C_i= \{j\in J\ ; \ \langle De_j,e_i\rangle>0\}$
 and $C^j= \{i\in I\ ;\  \langle De_j,e_i\rangle>0\}$. Note that since $D$ is doubly
stochastic, each $C_i$ and $C^j$ is non-empty and at most
countable. Moreover, $C_i\cap C_{i\prime}=\emptyset$ and $C^j\cap
C^{j\prime}=\emptyset$ for distinct $i, i^{\prime}\in I$ and
distinct $j,j^{\prime}\in J$. Hence
$$|I|\leq |C|\leq \aleph_0 \times |I|~~, ~~|J|\leq |C|\leq \aleph_0 \times |J|$$
where $\aleph_0$ is the cardinal number of $\mathbb N$. Since
$|I|,|J|\geq \aleph_0$, we have also $\aleph_0 \times |I|=|I|$ and
$\aleph_0 \times |J|=|J|$. Therefore $|I|=|C|=|J|$.\\

 Conversely,
let $\t:J\r I$ be a bijection. If $D:\lpj\r \lpi$ is defined for each $f=\S_{j\in J}f(j) e_j \in \lpj$ by
$D f=\S_{j\in J} f(j) e_{_{\t(j)}}$, then it is easily verified that $D$  is doubly
stochastic.
\end{proof}
\\

Since in this paper we are going to work with doubly stochastic operators, according to the previous theorem, we may assume that $I=J$.
The set of all row stochastic,
 column stochastic,  doubly stochastic operators and permutation maps on $\ell^p(I)$  are denoted, respectively, by
$\mathcal{RS}\big(\ell^p(I)\big), ~\mathcal{CS}\big(\ell^p(I)\big)$, ~$\ds$ and
$\mathcal P\big(\lpi\big)$. It is easily seen that $\mathcal P\big(\lpi\big)\subset \ds$. To obtain an essential property of these sets of operators, we need the following lemma. \\

\begin{lem}\label{lem2.3}
Let $p\in [1,+\infty)$ and $A:\ell^p(I)\rightarrow \ell^p(I)$
be a positive bounded linear operator. Then
\begin{enumerate}
\item[{\bf(i)}] $A$ is row stochastic if and only if
\begin{equation}\label{lem1.2eq1}
\forall f\in \ell^1(I),\qquad\sum_{j\in I} \langle Ae_j,f\rangle =\sum_{i\in I} f(i)
\end{equation}
\item[{\bf(ii)}] $A$ is column stochastic if and only if
\begin{equation}\label{lem1.2eq2}
\forall f\in \ell^1(I),\qquad \sum_{i\in I}\langle Af,e_i\rangle=\sum_{i\in I} f(i)
\end{equation}
\end{enumerate}
\end{lem}
\begin{proof}
 (i) Let $A:\ell^p(I)\rightarrow \ell^p(I)$ be row stochastic. Suppose $q\in (1,+\infty]$
is the exponent conjugate of $p$. Using the inclusion $\ell^1(I)\subset \ell^q(I)$, if $f\in \ell^1(I)$ then the map
$\langle\cdot,f\rangle:\ell^p(I)\rightarrow \reals$ is a bounded
linear functional. Moreover, if $ f=\sum_{i\in I}f(i) e_i$ then
 $\langle\cdot,f\rangle=\sum_{i\in I}
f(i)\langle\cdot, e_i\rangle$. To prove this last equality, it
suffices to consider $\ell^p(I)$ as a subset of
$\ell^{\infty}(I)=\big(\ell^1(I)\big)^*$.\\[2mm]

 Since
\[ \sum_{i\in I}\sum_{j\in I}|f(i)|\langle Ae_j
,e_i\rangle=\sum_{i\in I}|f(i)|<+\infty,\]
 by  Fubini's Theorem, we have
 \[\sum_{j\in I}\langle Ae_j,f\rangle=
 \sum_{j\in I}\sum_{i\in I}f(i)\langle Ae_j ,e_i\rangle=
 \sum_{i\in I}\sum_{j\in I}f(i)\langle Ae_j ,e_i\rangle=\sum_{i\in
 I}f(i)\]
 The converse is clear.\\
 (ii) Suppose $A$ is column stochastic. Let $A^*:\ell^q(I)\rightarrow \ell^q(I)$ be the adjoint
  map. It is easily seen that $A^*$ is row stochastic. Hence, by part (i),
 \[\forall f\in \ell^1(I),\qquad \sum_{i\in I}\langle e_i,Af\rangle=\sum_{i\in I}
 \langle A^*e_i,f\rangle=\sum_{i\in I} f(i)\]
 \end{proof}
 \begin{thm}\label{thm2.4}
 If $A$ and $B$ belong to $\mathcal{RS}\big(\ell^p(I)\big)$ then so does
 $AB$, i.e. the set $\mathcal{RS}\big(\ell^p(I)\big)$ is closed under
 combination. The same conclusion holds  for  sets
 $\mathcal{CS}\big(\ell^p(I)\big)$ and $\ds$.
 \end{thm}
 Proof. Let $A,B\in \mathcal{RS}\big(\ell^p(I)\big)$ and suppose  $A^*$ is the adjoint of $A$. Then, using Lemma \ref{lem2.3}, for $i\in I$ we have
 \begin{eqnarray*}
 \sum_{j\in I}\langle ABe_j,e_i\rangle &=&\sum_{j\in I}\langle
 Be_j,A^*e_i\rangle\\
  &=&\sum_{r\in I}\langle A^*e_i,e_r\rangle\\
   &=&\sum_{r\in I} \langle e_i,Ae_r\rangle=1
   \end{eqnarray*}
 i.e. $AB\in \mathcal{RS}\big(\ell^p(I)\big)$.\hfill$\Box$

\begin{lem}\label{lem norm ds}
If $D\in \ds$, then $\|D\|\leq 1$.
 \end{lem}
 Proof.
 For $f=\sum_{j\in I}f(j)e_j\in \ell^p(I)$, using the continuity
 of $D$, we have
 \[ Df=\sum_{j\in I}f(j)De_j\]
 Hence
 \begin{eqnarray*}
 \|Df\|_p^p&=& \sum_{i\in I}\big|\langle D f,e_i\rangle\big|^p\\
           &=& \sum_{i\in I}\big|\sum_{j\in I}f(j)\langle De_j,
           e_i\rangle\big|^p\\
           &\leq & \sum_{i\in I}\sum_{j\in I}|f(j)|^p\langle De_j,
           e_i\rangle
 \end{eqnarray*}
  The last inequality
 has been resulted from Jensen's inequality and the fact that $D$ is row stochastic.
  Now changing the
 order of summation, and using the fact that $D$ is also column
 stochastic, we have
 \[ \|Df\|_p^p\leq \sum_{j\in I}|f(j)|^p\sum_{i\in I}\langle
 De_j,e_i\rangle=\|f\|_p^p\]
  from which the result follows.\hfill$\Box$
\\

The following proposition, which presents a simple way to construct doubly stochastic operators, will be used in next sections.

 \begin{prop}\label{prop2.6}
 Let $I$ be a non-empty set and $p\in[1,\infty)$. Then
 corresponding to a family of non-negative real numbers $\{ d_{ij}\ ;\  i,j\in I\}$
 with
\begin{eqnarray}\label{prop2.6eq1}
 \forall i\in I, ~~~ \sum_{j\in I}d_{ij}=1,\qquad \forall j\in I,
~~~ \sum_{i\in I}d_{ij}=1
\end{eqnarray}
there exists a unique doubly stochastic operator $D$ on
$\ell^p(I)$ such that
\[\langle De_j,e_i\rangle=d_{ij}.\]
 \end{prop}
\begin{proof}
Let $\{d_{ij};i,j\in I\}$ be a family of non-negative and real
numbers which satisfy  (\ref{prop2.6eq1}) and suppose $f=\sum_{j\in I}f(j)
e_j$ is any arbitrary element of $\lpi$. For $1\leq p<\infty$,
from Jensen's inequality, we have
\begin{eqnarray*}
|\sum_{j\in I} f(j) d_{ij}|^p  \leq  \sum_{j\in I} |f(j)|^p d_{ij}
\end{eqnarray*}
which holds for each $i\in I$. Thus
\[
\sum_{i\in I}|\sum_{j\in I} f(j) d_{ij}|^p  \leq \sum_{i\in I}
\sum_{j\in I} |f(j)|^p d_{ij} = \sum_{j\in I} \sum_{i\in I} |f(j)|^p
d_{ij} = \|f\|^p.
\]
Hence the linear operator $D:\lpi\r\lpi$ defined by
$$Df=\sum_{i\in I}\Big(\sum_{j\in I}f(j) d_{ij}\Big) e_i$$
is bounded.\\
 Since $\langle De_j,e_i\rangle=d_{ij}$ for each $i,j\in I$, by assumption we have $D\in \ds$.\\
 To show the  uniqueness of $D$, suppose $A:\lpi\r\lpi$ is a bounded linear operator which satisfies
 $\langle Ae_j,e_i \rangle =d_{ij}$, for all $i,j \in I$.
  For each $i\in I$, and $f=\sum_{j\in I}f(j) e_j \in \lpi$,
 we have
$$(Af)(i)=\langle Af,e_i\rangle =\sum_{j\in I}f(j) \langle Ae_j,e_i\rangle= \sum_{j\in I}d_{ij}f(j) = \langle Df,e_i \rangle =(Df)(i).$$
Thus $A=D$.
\end{proof}
\section{Majorization on $\ell^p(I)$}

As was pointed out in the Introduction, the notion of majorization
in finite dimension has several equivalents, each of which can be
used to extend this theory to more general spaces. Here, we take
the approach based on the doubly stochastic operators.
\begin{defn}
For two elements $f,g\in \ell^p(I)$, we say $f$ is  majorized by
$g$ (or $g$ majorizes $f$), and denote it by $f\prec g$, if there
exists a doubly stochastic operator $D\in
\mathcal{DS}\Big(\ell^p(I)\Big)$ such that $f=Dg$.
\end{defn}
In order to obtain some consequences of this definition we need
the following lemma.

\begin{lem}\label{lem3.2}
Let $-\infty \leq a<b \leq +\infty$ and
$\phi:(a,b)\rightarrow [0,+\infty)$ be a convex function. For
$f,g\in \ell^p(I)$ with $Im(f), Im(g)\subseteq (a,b)$, if
 $f\prec g$  then
 \begin{equation}
 \sum_{i\in I}\phi(f_i)\leq \sum_{i\in I}\phi(g_i),
 \end{equation}
 where $f_i=f(i)$ and $g_i=g(i) $ for
 all $i\in I$.
\end{lem}
\begin{proof}
Suppose $f=Dg$, for some $D\in \mathcal{DS}\Big(\ell^p(I)\Big)$. Hence, for each $i\in I$,
\[ f_i=\langle f,e_i\rangle =\sum_{j\in I} \langle De_j,e_i\rangle g_j\]
Since $\phi$ is continuous and convex we will obtain
\[\phi(f_i)\leq \sum_{j\in I} \langle De_j,e_i\rangle \phi(g_j)\]
Thus
\begin{eqnarray*}
\sum_{i\in I} \phi(f_i) &\leq & \sum_{i\in I}\sum_{j\in I} \langle De_j,e_i\rangle \phi(g_j)\\
&=& \sum_{j\in I}\sum_{i\in I} \langle De_j,e_i\rangle \phi(g_j)\\
&=& \sum_{j\in I} \phi(g_j)
\end{eqnarray*}
\end{proof}
\begin{cor}\label{cor3.3}
For  $f,g\in \ell^p(I)$, if $f\prec g$ and $g\prec f$ then
 \begin{equation}
 \S_{i\in I}\phi(f_i) = \S_{i\in I}\phi(g_i),
 \end{equation}
for every convex function $\phi:(a,b)\rightarrow [0,+\infty)$ with
$Im(f), Im(g)\subseteq (a,b)$.
\end{cor}

\noindent
It must be noted that the converse of this corollary is not true
in general.

\begin{ex}\label{ex f=sum 1/2^n  e_n}
 {\rm
For $f=\sum_{n\in \mathbb N}{1\over 2^n}e_{n+1}$ and $ g=\sum_{n\in
\mathbb N}{1\over 2^n}e_n$ in $ \lpn$, let $\phi:(a,b)\r [0,\infty)$
be a convex function with $Im(f),Im(g)\subseteq (a,b)$.
First, suppose $\phi (0)>0$. Then
$$\lim_{n\r\infty}\phi
(f_n)=\lim_{n\r\infty}\phi (g_n)=\phi (0)>0 ,$$
which shows that
$$\sum_{n\in \mathbb N}\phi (f_n)=\sum_{n\in \mathbb N}\phi (g_n)=+\infty.$$
If $\phi(0)=0$ then
$$\sum_{n\in \mathbb N}\phi (f_n)=\sum_{n\in \mathbb N}\phi ({1\over 2^n})=\sum_{n\in \mathbb N}\phi (g_n)<+\infty.$$
Hence for every convex function $\phi:(a,b)\r [0,+\infty)$ we have $\sum_{n\in \mathbb N}\phi (f_n)=\sum_{n\in \mathbb N}\phi (g_n)$.

Now if for some doubly stochastic $D\in \ds$,~ $f=Dg$, then the
equality
$$0=f_1=\sum_{n\in \mathbb N}\langle De_n,e_1\rangle g_n = \sum_{n\in \mathbb N}{\langle De_n,e_1\rangle \over 2^n}$$
 implies  $\langle De_n,e_1\rangle =0$, for all $n\in \mathbb N$.
 Thus $\sum_{n\in \mathbb N}\langle De_n,e_1\rangle =0$ which
 contradicts the fact that $D$ is doubly stochastic. Hence $f\not\prec g$. A similar argument shows even $g\not\prec f$.
 \rm}
\end{ex}
The following theorem, which is our main result in this section,
 will play a crucial rule in the next section.
\begin{thm}\label{thm3.4}
For  $f,g\in \ell^p(I)$ the following conditions are equivalent.
 \begin{itemize}
\item[{\rm(1)}]$f\prec g$ and $g\prec f$.
\item[{\rm(2)}] there
is a permutation $P\in\mathcal P\Big(\ell^p(I)\Big)$
 such that $f=P g$.
\end{itemize}
\end{thm}
\begin{proof}
For each $f\in \lpi$ let $I_f^+, I_f^0$ and $I_f^-$
 be defined as follows.
 \begin{eqnarray*}
 I_f^+=\{i\in I;f(i)>0\},\\
  I_f^0=\{i\in I;f(i)=0\},\\
  I_f^-=\{i\in I;f(i)<0\}.
 \end{eqnarray*}
It is clear that both $\ifp$ and $\ifm$ are at most countable.
Let $\{\ifn; n\in \mathbb N\}$ be a family of subsets of $\ifp$
defined inductively as follows.
$$I_f^1 :=\Big\{ i\in \ifp; f(i)=\max\{f(j); j\in\ifp\}  \Big\}$$
and for $n>1$,
$$I_f^n :=\Big\{ i\in \ifp; f(i)=\max\{f(j); j\in\ifp\smallsetminus \v\bigcup_{k=1}^{n-1} \ifk\}  \Big\}$$
It is easily seen that $\ifn$ is at most a finite set, and that if
$\ifp$ is infinite, then $\ifn\neq \emptyset$, for each $n\in
\mathbb N$. Moreover, the family  $\{\ifn; n\in \mathbb N\}$ is
mutually disjoint and $\ifp =\v\cup_{n\in \mathbb N} I_f^n$. For
$n\in \mathbb N$ with $\ifn\neq \emptyset$, let $f_n>0$ be the
value of $f$ on $\ifn$. If $\ifn = \emptyset$ then we define $f_n$
equal $0$. It is clear that for $m,n\in \mathbb N$ with $\ifn$ and
$I_f^m$ non-empty, if $n<m$ then $f_m<f_n$.

 Now for $f,g\in \lpi$,
suppose $f\prec g$ and $ g\prec f$. Let $\phi_c:\mathbb R\r
\mathbb R$ be a convex function defined by
$\phi_c(x)=(x-c)\chi_{_{[c,\infty)}}(x)$, with $c\in \mathbb
R$. By Corollary \ref{cor3.3},
\begin{equation}\label{thm3.4eq1}
\sum_{i\in I} \phi_c\big(f(i)\big)=\sum_{i\in
I}\phi_c\big(g(i)\big)
\end{equation}
for each $c\in \mathbb R$. For $c=0$, we have
\begin{equation}\label{thm3.4eq2}
\sum_{i\in \ifp} \phi_0\big(f(i)\big)=\sum_{i\in I}\phi_0\big(f(i)\big)=\sum_{i\in I}\phi_0\big(g(i)\big) =\sum_{i\in \igp} \phi_0\big(g(i)\big)
\end{equation}
which shows that $\ifp \neq \emptyset$ if and only if
 $\igp\neq\emptyset$. Suppose $\ifp\neq \emptyset$.
Using induction,  we show that for each $n\in \mathbb N$,
\begin{itemize}
 \item[{\rm(i)}] $f_n=g_n$,
  \item[{\rm(ii)}] $|I_f^n|= |I_g^n|$.
 \end{itemize}

For $n=1$, if $I_f^1=\emptyset$, then $f\leq 0$  and therefore
$\ifp=\emptyset$ which is contrary to our assumption. Hence
$I_f^1\neq\emptyset$. Similarly $I_g^1\neq \emptyset$. Suppose
$f_1\neq g_1$ and, for example,  $f_1<g_1$. Then for each $i_1\in
I_f^1$ and $i_2\in I_g^1$, $f(i_1)=f_1<g_1=g(i_2)$. Using the
convex function $\phi_c$ with $c=\min\{f_1,g_1\}$, we have
$$\S_{i\in
I}\phi_c\Big(f(i)\Big)=0 <     g_1-f_1   \leq \S_{i\in I}
\phi_c\Big(g(i)\Big)$$ which contradicts (\ref{thm3.4eq1}). Hence
$f_1=g_1$. Again, taking $c=\max\{f_2,g_2\}$ in  (\ref{thm3.4eq1}),
we have
$$(f_1-c)|I_f^1|=\sum_{i\in I} \phi_c\Big(f(i)\Big)=\sum_{i\in I} \phi_c\Big(g(i)\Big) = (g_1-c)|I_g^1|.$$
Hence ${\rm (i)}$ and ${\rm(ii)}$ are satisfied for $n=1$.\\
 Suppose ${\rm (i)}$ and ${\rm (ii)}$ hold for each
 $k=1,\dots,  n$. If $I_f^{n+1} =\emptyset$ then $I_f^j=\emptyset$
 for all $j\geq n+1$. Hence, using once more
 equation (\ref{thm3.4eq2}), we will have
 \[ \sum_{k=1}^n f_k  |I_f^k| = \sum_{i\in I} \phi_0\Big(f(i)\Big) =
  \sum_{i\in I} \phi_0\Big(g(i)\Big) \geq \sum_{k=1}^{n+1} g_k
  |I_g^k|\]
  which implies that the term $g_{n+1}|I_g^{n+1}|$ is non-positive. Hence
  $I_g^{n+1}=\emptyset$. In this case, $f_{n+1}=g_{n+1}=0$, i.e.
  (i) and (ii) are satisfied for $n+1$. If $I_f^{n+1}
  \neq\emptyset$ then the same argument shows that $I_g^{n+1}\neq
  \emptyset$.  In this case, a similar procedure to that of $n=1$, once with $c=\min\{f_{n+1},g_{n+1}\}$ and then
  with   $c=\max\{f_{n+2},g_{n+2}\}$ in (\ref{thm3.4eq1}), implies (i) and (ii)
for $n+1$.

 By (ii), there is a bijection $\theta _n:I_g^n\r I_f^n$ for each
 $n\in \mathbb N$ with $I_f^n\neq\emptyset$.
 Now
we can define a bijection $\theta^+:I_g^+=\cup_{n\in \mathbb
N}I_g^n\r I_f^+$ given by $\theta^+(j)=\theta_n(j)$ if $j\in
I_g^n$.\\

Let $D$ be a doubly stochastic operator on $\lpi$ satisfying
$f=Dg$. For simplicity, we let $d_{ij}:=\langle De_j,e_i\rangle$, for each $i,j\in I$. We show that if $I_f^n\neq \emptyset$ then
\begin{equation}\label{eq ii}
\forall i\in\ifn,\qquad \sum_{j\in\ign}\dij=1
\end{equation}
and
\begin{equation}\label{eq iii}
\forall j\in \ign,\qquad \sum_{i\in\ifn}\dij=1
\end{equation}
To prove (\ref{eq ii}) and (\ref{eq iii}), first
suppose $n=1$ and that $I_f^1\neq \emptyset$. We show that
$\sum_{j\notin I_g^1}d_{ij}=0$, for all $i\in I_f^1$, which then implies
(\ref{eq ii}) for $n=1$. If for some $i\in I_f^1$,~ $\sum_{j\notin
I_g^1}d_{ij}>0$, then
$$0<f_1=f(i)=\sum_{j\in I}d_{ij}g(j)=\sum_{j\in I_g^1}d_{ij}g_1+\sum_{j\notin I_g^1}d_{ij}g(j)<\sum_{j\in I_g^1}d_{ij}g_1+\sum_{j\notin I_g^1}d_{ij}g_1=g_1$$
which contradicts the fact that $f_1=g_1$. Hence we have shown that
 $d_{ij}=0$, for each $i\in
I_f^1$ and $j\notin I_g^1$, and therefore, $\sum_{j\in I_g^1}d_{ij}=1$.

To see (\ref{eq iii}) for $n=1$, suppose there exists $j\in I_g^1$ with $\sum_{i\in
I_f^1}d_{ij}<1$, then
$$|I_f^1|=\sum_{i\in I_f^1}\sum_{j\in I_g^1}d_{ij}=\sum_{j\in I_g^1}\sum_{i\in I_f^1}d_{ij}<|I_g^1|,$$
which contradicts (ii) for $n=1$.\\
By induction and using a similar method, we will see that (\ref{eq ii}) and (\ref{eq iii}) hold
for each $n\in \mathbb N$.

An immediate consequence of the above facts is that,
\begin{itemize}
 \item[{\rm(i)}]  $\forall i\in I_f^+~~\forall j\in I\setminus I_g^+,~~~\dij=0$
 \item[{\rm(ii)}] $\forall j\in I_g^+~~\forall i\in I\setminus I_f^+,~~~\dij=0$
 \end{itemize}

Replacing $f$ and $g$ by $-f$ and $-g$ and noting that $I_f^-=I_{-f}^+$ and $I_g^-=I_{-g}^+$, we obtain a similar
bijection $\theta^-:I_g^-\r I_f^-$, and the following results.
\begin{itemize}
 \item[{\rm(iii)}]  $\forall i\in I_f^-~~\forall j\in I\smallsetminus I_g^-,~~~\dij=0$
 \item[{\rm(iv)}] $\forall j\in I_g^-~~\forall i\in I\smallsetminus I_f^-,~~~\dij=0$
 \end{itemize}

Using all above facts, it is easily verified that if $D_0:\ell^p(I_g^0)\r\ell^p(I_f^0)$ is
the map defined by $\langle D_0e_j,e_i\rangle=\dij$, for $i\in
I_f^0$ and   $j\in I_g^0$, then it is doubly stochastic. By
Theorem \ref{|I|=|J|},
 $|I_f^0|=|I_g^0|$, i.e.
there exists a bijection $\theta^0:I_g^0\r I_f^0$. Now we
define a bijection $\theta:I\r I$ by

\begin{eqnarray*}
\begin{array}{l}
\theta(j)=\left\{
\begin{array}{ll}
\theta^+(j)    &j\in I_g^+\\[.3cm]
 \theta^-(j)    &j\in I_g^-\\[.3cm]
\theta^0(j)&j\in I_g^0
\end{array}
\right.
\end{array}
\end{eqnarray*}

Let $P$ be the permutation on $\lpi$ corresponding to $\theta$ . We
claim  that $f=P g$. To see this, note that for each $i\in I$,
$$\big(P g\big)(i)=\langle\sum_{j\in I}
g(j)e_{_{\theta(j)}}, e_i\rangle(i)=g\big(\theta^{-1}(i)\big)$$
If $i\in I_f^+$ then there exists $n\in \mathbb N$ such that $i\in I_f^n$, and therefore $\theta^{-1}(i)\in I_g^n$.
Hence $g\big(\theta^{-1}(i)\big)=g_n=f_n=f(i)$. A similar argument holds if $i\in I_f^-$. Finally, if $i\in I_f^0$, then $\theta^{-1}(i)\in I_g^0$. Hence $g\big(\theta^{-1}(i)\big)=0=f(i)$.
Thus, $Pg(i)=f(i)$, for all $i\in I$, i.e. $f=Pg$.

The converse is evident.
\end{proof}


\section{ Linear Maps  Preserving  Majorization }

 In this section, we characterize  bounded linear operators on $\lpi$, with $p\in (1,+\infty)$ and  $I$ an infinite set, which preserves the majorization relation.
\begin{defn}\label{defn4.1}
A bounded linear operator $T:\ell^p(I)\rightarrow \ell^p(I)$ is
called a majorization preserver on $\ell^p(I)$, if $T$ preserves
the majorization relation, i.e. for $f,g\in \ell^p(I)$, $f\prec g$
implies $Tf\prec Tg$. We denote  by $\Mpr$  the set of all such
operators.
\end{defn}

In order to have some examples of this class of operators, we need first some preliminaries.
It is easily seen that for $\a \in \mathbb R$ and $S,T\in \Mpr$,~
$\a T, ST \in \Mpr$, i.e. $\Mpr$ is closed under the scalar
multiplication and combination. We will see later that this set  is not closed under addition.

For a one-to-one map $\s:I\r I$,  let $P_{\s}:\lpi\r\lpi$ be defined for each for $f=\sum_{j\in I}f_j
e_j\in \lpi$ by
$P_{\s}(f)=\sum_{j\in I}f_j e_{\s (j)}$. Clearly, $P_{\s}$ is a bounded linear operator with $\|P_{\s}\|\leq 1$. Note that if, in addition, $\s :I\r I$ is on-to then $P_{\s}$ is a permutation.

\begin{lem}\label{d tilde}
Let $D:\lpi\r\lpi$ be a doubly stochastic operator and $\Sigma$ be any family of one-to-one maps from $I$ to $I$
which satisfies $\s_1(I)\cap \s_2(I) =\emptyset$, for distinct
$\s_1,\s_2\in \Sigma$. Then
there exists a doubly stochastic $\widetilde{D}\in \ds$ such that
$P_{\s} D = \widetilde{D}P_{\s}$, for all $\s\in \Sigma$.
\end{lem}
\begin{proof}
For $i,j\in I$, let $d_{ij}:=\langle De_j,e_i\rangle$ and suppose  $\tilde{d}_{ij}$ is defined by
  \begin{equation}\label{dtildeij1}
\tilde{d}_{ij}=\left\{
\begin{array}{lll}
d_{\s^{-1}(i) \s^{-1}(j)}    &\mbox{if}&  i,j\in\s(I)~~(\ \mbox{for some}\  \s\in \Sigma)\\[.2cm]
~~0~~                                &\mbox{if}&  i\in \s(I),~j\notin\s(I)~(\ \mbox{for some}\  \s\in \Sigma)\\[.2cm]
~~1~~                              &\mbox{if}&  i\notin \cup_{\s\in\Sigma} \s(I),~j=i\\[.2cm]
~~0~~                            &\mbox{if}&  i\notin \cup_{\s\in\Sigma} \s(I),~j\neq i\\
\end{array}
\right.
\end{equation}
 By considering the two cases $i\in \s(I)$, for some
$\s\in\Sigma$, and  $i\notin \cup_{\s\in\Sigma} \s(I)$, it is easy
to see that $\sum_{j\in I} \tilde{d}_{ij} =1$ for each $i\in I$.
Similarly, writing (\ref{dtildeij1}) in the following form,
  \begin{equation*}
\tilde{d}_{ij}=\left\{
\begin{array}{lll}
d_{\s^{-1}(i) \s^{-1}(j)}    ~~&\mbox{if}&j,i\in\s(I)~(\ \mbox{for some}\ \s\in \Sigma)\\[.2cm]
~~0~~                                ~~&\mbox{if}&j\in \s(I),~i\notin\s(I)~(\ \mbox{for some}\ \s\in \Sigma)\\[.2cm]
~~1~~                              ~~&\mbox{if}&j\notin \cup_{\s\in\Sigma} \s(I),~i=j\\[.2cm]
~~0~~                            ~~&\mbox{if}&j\notin \cup_{\s\in\Sigma} \s(I),~i\neq j,\\
\end{array}
\right.
\end{equation*}
it is seen that $\sum_{i\in I} \tilde{d}_{ij} =1$ for each $j\in
I$. Hence, using Proposition $\ref{prop2.6}$, there exists a doubly stochastic
operator $\widetilde{D}:\lpi\r\lpi$ which satisfies $\langle \tilde{D}
e_j,  e_i \rangle = \tilde{d}_{ij}$ for all $i,j \in I$.

 It remains to show that $ P_{\s}D = \widetilde{D}P_{\s}$, for each
$\s\in\Sigma$. We have
$$\widetilde{D}P_{\s}(e_j) = \widetilde{D}(e_{_{\!\s(j)}})  = \S_{i\in \s(I)} \tilde{d}_{i\s(j)}e_i  =
\S_{i\in \s(I)} d_{\s^{-1}(i)j}e_i  = \S _{r\in I} d_{rj}
e_{_{\!\s(r)}}$$
 and $$P_{\s}D(e_j)= P_{\s}\Big( \S_{i\in I}d_{ij} e_i \Big )= \S_{i\in I}d_{ij} e_{_{\!\s(i)}}.$$
Hence
\begin{equation*}
  P_{\s}D(e_j) = \widetilde{D}P_{\s}(e_j),
\end{equation*}
for all $j\in I$. Thus
  $\widetilde{D}P_{\s}  =  P_{\s}\!D$,
  for each $\s\in \Sigma$.
\end{proof}
\begin{ex}\label{ex1}{\rm
Let $\s:I\r I$ be a one-to-one map. For $f,g\in \lpi$ suppose
$f\prec g$, i.e. $f=Dg$ for some $D\in \ds$. By Lemma \ref{d
tilde}, corresponding to the singleton $\Sigma=\{\s\}$, there exists $\widetilde{D}\in \ds$ for which
$P_{\s}D=\widetilde{D}P_\s$. Therefore $P_{\s}f=P_{\s}Dg=\widetilde{D}P_\s
g$, which shows that $P_{\s}f\prec P_{\s}g$. Thus each $P_{\s}$
preserves majorization. In particular each permutation belongs to
$\Mpr$, i.e.
\[
\Plpi\subseteq\Mpr.
\]
\rm}
\end{ex}

\begin{ex}\label{ex2}{\rm
For a fixed $k\in \mathbb N$, let $T:\lpn\r\lpn$ be an operator
defined for
$f=\sum_{n=1}^{\infty}f_ne_n\in \lpn$ by $T(f)=\sum_{n=1}^{\infty}f_{[\frac{n}{k}]}e_n$,  where $[\frac{n}{k}]$ denotes the
greatest integer contained in $\frac{n}{k}$, and $f_0:=0$. Then, $T$ is easily seen
to be linear and bounded (with $\|T\|=\sqrt[p]{k}$). Suppose
$\Sigma =\{\s_1,\dots ,\s_k\}$ where each $\s_i:\mathbb N\r\mathbb
N$ (for $1\leq i\leq k$) is a one-to-one map defined by
$\s_i(n)=nk+i-1$, for all $n\in \mathbb N$. It is easy to see that
$T=\sum_{i=1}^kP_{\s_i}$ and that the family $\Sigma$ satisfies
condition of Lemma \ref{d tilde}. If $f\prec g$ in $\lpi$, i.e.
$f=Dg$ for some $D\in \dslpn$, then Lemma \ref{d tilde} implies
that there exists  a doubly stochastic $\tilde{D}\in {\mathcal
DS}\big (\lpn\big )$ for which $\tilde{D}P_{\s_i}=P_{\s_i}D$ for
$i=1,\dots,k.$ Hence
$$Tf=\sum_{i=1}^k P_{\s_i}f=\sum_{i=1}^k P_{\s_i}Dg=\sum_{i=1}^k \tilde{D}P_{\s_i}g=\tilde{D}Tg$$
i.e. $T$ preserves majorization. \rm}
 \end{ex}

 In the following theorem, which is a generalization of Example \ref{ex2}, we construct a family of bounded linear operators which preserve majorization. As we will see in Theorem \ref{thm4.9}, in the case $1<p<+\infty$, every  majorization preserver will also be in this form.
\begin{thm}\label{thm4.5}
Let $p\in [1,+\infty)$, $I$ be an infinite set and $I_0\subset I$ be a countable subset. Moreover, suppose $\Sigma=\{\s_i:I\r I\ ;\ i\in I_0\}$ is a family of one-to-one maps such that
 for all  $i_1,i_2 \in I_0$ with $i_1\neq i_2$,
 $\sigma_{i_1}(I)\cap \sigma_{i_2}(I)=\emptyset$. If $(\alpha_i)_{i\in I_0}$ is an element of $\ell^p(I_0)$ then $T:=\sum_{i\in I_0}\alpha_i P_{\s_i}$ is a majorization preserver.
\end{thm}
\begin{proof}
It is easily seen that   $T=\sum_{i\in I_0}\a_i P_{\s_i}$ is a well-defined bounded linear map.
Suppose
 $f\prec g$, for $f,g \in \lpi$, and therefore $f=Dg$  for some  $D\in \ds$.\\
 Corresponding to the family $\{\sigma_i:I\r I;~i\in I_0\}$, let $\widetilde{D}\in \ds$  be the operator given by Lemma
 \ref{d tilde}. Then
\begin{eqnarray*}
\widetilde{D}(T g)    &=&    \widetilde{D}\Big(\S_{i\in I_0} \a_i P_{\s_i}(g)\Big)\\
           &=&    \S_{i\in I_0} \a_i \widetilde{D} P_{\s_i}(g)\\
           &=&    \S_{i\in I_0} \a_i P_{\s_i}D(g)\\
                      &=&    \S_{i\in I_0} \a_i P_{\s_i}(f)\\
                      &=&    Tf\\
\end{eqnarray*}
Hence $Tf\prec Tg$.
\end{proof}
\\

As was pointed out, the converse of this theorem is also true for $p\in (1,+\infty)$. In order to prove it, we need some lemmas.

\begin{lem}\label{alpha}
Let $a,b\in \mathbb R$ and $\{a_i;i\in I\}\ ,\ \{b_i;i\in I\}$ be two
families of real numbers, where $I$ is assumed to be a countable indexed set. If
$$\a a+\b b\in \{\a a_i+\b b_i;i\in I\},$$
for all $\alpha,\beta \in \mathbb R$, then there exists $i\in I$
such that $a=a_i$ and $b=b_i$.
\end{lem}
\begin{proof}
Let $C:=\{(\a,\b)\in \mathbb R ^2\ ;\  \a,\b>0,\ \a^2+\b^2=1\}$. Then,
by assumption, for each $(\a,\b)\in C$ there exists
$i=i_{(\a,\b)}\in I$ for which
$$\a a+\b b =\a a_i+\b b_i$$
Since $I$ is countable and $C$ is uncountable there exists two
distinct elements $(\a_1,\b_1),(\a_2,\b_2)\in C$ with
$i_{(\a_1,\b_1)}=i_{(\a_2,\b_2)}$, which for simplicity we denote it
by $i$ itself. Hence
\begin{eqnarray}\label{equation}
\begin{array}{l}
\left\{
\begin{array}{ll}
\a_1 a+\b_1 b=\a_1 a_i+\b_1 b_i\\[.3cm]
\a_2 a+\b_2 b=\a_2 a_i+\b_2 b_i\\
\end{array}
\right.
\end{array}
\end{eqnarray}

Note that any two distinct elements of $C$ are linearly
independent. Hence (\ref{equation}) implies that $a=a_i$ and
$b=b_i$.
\end{proof}
\\

As the following lemma shows, if $T:\lpi\r\lpi$ is a linear majorization preserver then, roughly speaking, each row of $T$ contains, at most, one non-zero element. In what follows, for $f,g\in\lpi$, we use  the notation $f\sim g$ whenever each of $f$ and $g$ is majorized by the other i.e. $f\prec g$ and $g\prec f$.
\begin{lem}\label{atmost}
Let $I$ be any infinite set, $p\in(1,\infty)$, and $T \in \Mpr$.
Then for any $i\in I$, there is at most one $j\in I$ such that
$\langle Te_j,e_i\rangle \neq 0$.
\end{lem}
\begin{proof}
Suppose, on the contrary, there exists $i_1\in I$ and two different
elements $j_1,j_2\in I$ for which
$$\langle Te_{j_1},e_{i_1}\rangle\neq 0~,~\langle Te_{j_2},e_{i_1}\rangle\neq 0$$
For simplicity we denote $\langle Te_{j_1},e_{i_1}\rangle,\langle
Te_{j_2},e_{i_1}\rangle$, respectively, by $a,b$. Let $F$ be given
by
\begin{eqnarray*}
F=\{i\in I; \langle Te_{j_1},e_i\rangle=a\}.
\end{eqnarray*}
Then $F$ is non-empty, and the inequality
\begin{eqnarray*}
\begin{array}{ll}
\v\sum_{i\in F}|a|^p  &= \v\sum_{i\in F} |\langle
Te_{j_1},e_i\rangle|^p\\
                      &\leq \v\sum_{i\in I} |\langle
Te_{j_1},e_i\rangle|^p\\
                      &= \|Te_{j_1}\|^p
                      <\infty\\
\end{array}
\end{eqnarray*}
shows that $F$ is finite. On the other hand, for any given
$j\neq j_1$ and $\alpha,\beta\in \mathbb R$, since $\alpha e_{j_1}+\beta e_{j_2}\sim \alpha e_{j_1}+\beta e_j$, we have
$$\forall \alpha
,\beta \in \mathbb R
~~~ \alpha Te_{j_1}+\beta Te_{j_2}
 \sim \alpha Te_{j_1}+\beta Te_j$$
 Hence, by Theorem \ref{thm3.4},
$$\alpha a+\beta b \in \{\alpha\langle Te_{j_1},e_i\rangle+\beta\langle Te_j,e_i\rangle\ ;\ i\in I\}$$
for  $j\neq j_1$ and all $\alpha,\beta\in \mathbb R$.
 But the indexed set $I$ can be
replaced by a countable one. Hence, by Lemma $\ref{alpha}$, for each $j\in I\setminus \{
j_1\}$ there exists $i\in I$ such that
$$ \langle Te_{j_1},e_i\rangle=a,~\langle Te_j,e_i\rangle=b.$$
Thus $i\in F$. Since $I$ is infinite and $F$ is finite,  there
exists $i_0\in F$ and a sequence $(j_n)$ in $I$, with $j_m\neq
j_n$ for $m\neq n$, such that
$$\langle e_{j_n},T^*e_{i_0}\rangle=\langle Te_{j_n},e_{i_0}\rangle=b\neq0$$
for all $n\in \mathbb N$. This contradicts the fact that $e_{j_n}$
converges to $0$ in the weak topology of $\lpi$.
\end{proof}\\

Using the previous lemma, the next example shows that the sum of two majorization preservers need not be a preserver.
\begin{ex}\label{ex atmost}{\rm
Let $\s_1,\s_2 :\mathbb N \r \mathbb N$ be defined by
$\s_1(n)=2n,~\s_2(n)=n$, for each $n\in \mathbb N$. Then by Example \ref{ex1}, the maps $P_{\s_1}$ and $P_{\s_2}$ are both majorization preservers. Now suppose
$T:=P_{\s_1}+P_{\s_2}$.
Then, since
$$\langle Te_1,e_2\rangle=\langle Te_2,e_2\rangle=1,$$
by Lemma \ref{atmost}, $T$ is no longer a majorization preserver.
\rm}
\end{ex}

We now have the main result of this paper.
\begin{thm}\label{thm4.9}
Suppose I is  an infinite set  and $p\in (1,+\infty)$. For a
bounded linear operator $T$ on $\ell^p(I)$ the following
conditions are equivalent.
 \begin{itemize}
 \item[{\rm(i)}] $T$ is a majorization preserver.
 \item[{\rm(ii)}] For $f,g\in \ell^p(I)$, if $f \sim g$ then $Tf\sim
 Tg$. Furthermore, for any $i\in I$ there is at most one $j\in I$
for which $\langle Te_j,e_i\rangle \neq 0$.
\item[{\rm(iii)}] For any $j_1,j_2\in I$, $Te_{j_1} \sim Te_{j_2}$
 and for each $i\in I$ there is at most one $j\in I$
with $\langle Te_j,e_i\rangle \neq 0$.
\item[{\rm(iv)}] $T= \sum_{i\in
I_0}\alpha_i P_{\sigma_i}$, where $I_0$ a countable
subset of $I$, $(\alpha_i)_{i\in I_0}$ is an element of $\ell^p(I_0)$,
and $\{\s_i:I\r I\ ;\ i\in I_0\}$ is a family of one-to-one maps
 such that for all  $i_1,i_2 \in I_0$ with $i_1\neq i_2$,
 $\sigma_{i_1}(I)\cap \sigma_{i_2}(I)=\emptyset$.
  \end{itemize}
\end{thm}
\begin{proof}
Suppose $T$ is a non-zero bounded linear operator on $\lpi$.\\
(i)$\Rightarrow$ (ii) is obtained from Definition \ref{defn4.1} and Lemma \ref{atmost}.\\
(ii)$\Rightarrow$ (iii) is clear.\\
(iii)$\Rightarrow$ (iv). For $j\in I$ let  $I(j):=\{i\in I;\tij\neq
0\}$. According to (iii) for $j_1\neq j_2$,
\begin{equation}\label{ij intersect}
 I(j_1)\bigcap I(j_2) =\emptyset.
\end{equation}

 On the other hand, $T\neq 0$. So there exists $j_0\in I$
such that $Te_{j_0}\neq 0$. Hence $I(j_0)\neq \emptyset$.\\
Now for $j\in I$  with $j\neq j_0$, ~$Te_j\sim Te_{j_0}$. Let
$P_j:\lpi\r\lpi$ be the permutation given by Theorem \ref{thm3.4},
so that $Te_j=P_jTe_{j_0}$. Also let $\theta_j:I\r I$ be the
bijection corresponds to $P_j$ which is uniquely determined by
$P_j(e_i)=e_{_{\!\theta_j(i)}}$,
 for all $i\in I$.\\
Let $I_0:=I(j_0)$ which is obviously a countable subset of $I$,
$\sigma_i:I\r I$ be defined by $\sigma_i(j)=\theta_j(i)$, and
$\alpha_i:=\langle Te_{j_0}, e_i\rangle$, for $i\in I_0$.\\
Note that for $i,j\in I$
\begin{eqnarray*}\label{prod}
 \langle Te_j,e_{_{\!\theta_j(i)}}\rangle &=&  \langle
 Te_j,P_j(e_i)\rangle\\
                                     &=& \langle
                                     P_j^*Te_j,e_i\rangle\\
                                     &=& \langle
                                     P_j^{-1}Te_j,e_i\rangle.\\
\end{eqnarray*}
Since $P^{-1}Te_j=Te_{j_0}$, we have
\begin{equation}\label{tej}
\langle Te_j,e_{_{\!\theta_j(i)}}\rangle = \langle
Te_{j_0},e_i\rangle.
\end{equation}
for every $i,j\in I$.
 This shows that for each $i\in I_0=I(j_0)$,
$~\theta_j(i)\in I(j)$.
Hence for $i\in I_0$ and $j_1\neq j_2$, since
$\sigma_i(j_1)=\theta_{j_1}(i)\in I(j_1),$ and
$\sigma_i(j_2)=\theta_{j_2}(i)\in I(j_2)$, (\ref{ij intersect})
shows that $\sigma_i(j_1) \neq \sigma_i(j_2)$, i.e. $\sigma_i:I\r
I$ is
one-to-one.

Let $i_1,i_2$ are two distinct elements of $I_0$. We will show
that $\sigma_{i_1}$ and $\sigma_{i_2}$ have disjoint
ranges. Suppose, on the contrary, there exist $j_1,j_2 \in I$ for
which $\sigma_{i_1}(j_1)=\sigma_{i_2}(j_2)$ which implies that
\begin{eqnarray}\label{theta ij}
 \theta_{j_1}(i_1)=\theta_{j_2}(i_2).
\end{eqnarray}
By  $(\ref{prod})$, we have
\begin{eqnarray}\label{neq0}
\langle Te_{j_1}, e_{_{\!\theta_{j_1}(i_1)}} \rangle     = \langle
Te_{j_0}, e_{i_1} \rangle   \neq   0,
\end{eqnarray}
and
\begin{eqnarray}\label{neq1}
\langle Te_{j_2}, e_{_{\!\theta_{j_2}(i_2)}} \rangle   =
\langle Te_{j_0}, e_{i_2} \rangle   \neq   0.
\end{eqnarray}
By (\ref{theta ij}),
$e_{_{\!\t_{j_1}(i_1)}}=e_{_{\!\t_{j_2}(i_2)}}$. Hence
(\ref{neq0}), (\ref{neq1}) and the assumption of (iii) implies
that $j_1=j_2$, which, again by  (\ref{theta ij}), leads to the contradiction $i_1=i_2$.

Finally, we show that $\sum_{i\in I_0}\a_iP_{\s_i}$ converges
(unconditionally) in norm to $T$.  First we consider the case where
$I_0$ is infinite. For simplicity, suppose $I_0=\mathbb N$.
 We will show that
$\sum_{n=1}^{\infty}\a_nP_{\s_n}$ converges to $T$ in the norm
topology of $\mathcal{B}\Big(\lpi\Big)$. For $j\in I$, we have
\begin{eqnarray*}
Te_j             =       P_j(Te_{j_0})   &=&    P_j(\v\sum_{i\in I_0} \langle Te_{j_0},e_i\rangle e_i)\\
   &=&    \v\sum_{i\in I_0} \langle Te_{j_0},e_i\rangle P_j e_i\\
   &=&    \v\sum_{i\in I_0} \a_i e_{_{\!\s_i(j)}}\\
   &=& \v\sum_{n=1}^{\infty}\a_n e_{_{\!\s_n(j)}}
\end{eqnarray*}

Hence for  $f=\S_{j\in I}f_j e_j \in \lpi$, and $n\in
\mathbb N$,
\begin{eqnarray*}
\| T f-\v\S_{k=1}^n\a_k P_{\s_k}(f) \|^p  &=&  \|\S_{j\in I} f_j
Te_j-\S_{k=1}^n\S_{j\in I} \a_k f_j ~e_{_{\!\s_k(j)}} \|^p\\
&=& \|\S_{j\in I}\S_{k\in \mathbb N} \a_k f_j ~e_{_{\!\s_k(j)}} - \S
_{k=1}^n\S_{j\in I}
\a_k f_j ~e_{_{\!\s_k(j)}}  \|^p\\
&=&  \| \v \S _{k>n , j\in I}\a_k f_j ~e_{_{\!\s_k(j)}} \|^p \\
&=&   \v \S _{k>n , j\in I}|\a_k f_j|^p   = \|f\|^p \S_{k=n+1}^{\infty} |\a_k|^p\\
\end{eqnarray*}
Hence $\| T-\sum_{k=1}^n\a_k P_{\s_k} \| \leq \big(\sum_{k=n+1}^{\infty}
|\a_k|^p\big)^{\frac{1}{p}}\to 0$, as $n\to\infty$.\\
 (iv)$\Rightarrow$(i). This is
Theorem \ref{thm4.5}.
\end{proof}
\\

By Theorem \ref{thm4.5}, if $T:\ell^1(I)\r\ell^1(I)$ is in the form described in part (iv) of the previous theorem, then $T$ is a majorization preserver. However it should be noted that, as the following example shows, not every majorization preserver $T:\ell^1(I)\r\ell^1(I)$ is necessarily in this form.

\begin{ex}{\rm
Let $h=\sum_{j\in I} h_j e_j  \in  \ell ^1(I)$ be a non-zero element and suppose
$T_h:\ell^1(I)\r\ell^1(I)$ is defined,  for each $f=\sum _{i\in
I}f_i\in \ell^1(I)$, by $T_h(f)=\big(\sum _{i\in I}f_i\big )h$. It is easily seen that $T_h$
is linear and bounded (with $\|T_h\|=\|h\|$). On the other hand, for $f,g\in
\ell^1(I)$, if  $f\prec g$ then $\sum _{i\in I}f_i=\sum _{i\in
I}g_i$. Hence  $T_h(f)=T_h(g)$, which clearly implies that $T_h(f)\prec T_h(g)$. Thus  $T_h$ is a
majorization preserver, which is not in the form described in the previous theorem. }
\end{ex}

\end{document}